\documentclass[a4paper]{amsart}
\usepackage{xypic,amscd,amssymb,latexsym,srcltx}
\usepackage{amsthm}
\renewcommand{\proofname}{\textbf{Proof}}

\makeatletter
\renewenvironment{proof}[1][\proofname]{\par
  \pushQED{\qed}%
  \normalfont
  \topsep6\p@\@plus6\p@\relax
  \trivlist
  \item[\hskip\labelsep
        \bfseries #1\@addpunct{.}]\ignorespaces
}{%
  \popQED\endtrivlist\@endpefalse
}
\makeatother
\usepackage{amsmath}  
\usepackage{ulem}
\usepackage{tikz-cd} 
\usepackage{enumitem}
\def\@setcopyright{}
\makeatother

\newcommand{\Syl}{\operatorname{Syl}\nolimits}

\newcommand{\Aut}{\operatorname{Aut}\nolimits}

\newcommand{\GL}{\operatorname{GL}\nolimits}
\newcommand{\GU}{\operatorname{GU}\nolimits}

\newcommand{\SL}{\operatorname{SL}\nolimits}
\newcommand{\SU}{\operatorname{SU}\nolimits}
\newcommand{\GSp}{\operatorname{GSp}\nolimits}

\newtheorem{lemma}{Lemma}[section]
\newtheorem{proposition}[lemma]{Proposition}

\newtheorem{theorem}[lemma]{Theorem}
\newtheorem{remark}[lemma]{Remark}
\newtheorem{conjecture}[lemma]{Conjecture}
\newtheorem*{acknowledgments*}{ACKNOWLEDGMENTS}
\newtheorem{definition}[lemma]{Definition}

\newtheorem*{notation*}{Notation}

\newtheorem*{conj*}{Conjecture}
\newtheorem*{thm*}{Main Theorem}
\newtheorem*{remark*}{Remark}
\newtheorem*{lemma*}{Lemma}
\newtheorem*{example*}{[Example]}

\title{Block-exoticity of reduction simple fusion systems on $S(n,p)$}

\author{Jun Liao, Yikun Liu}

\address{School of Mathematics and Statistics, Hubei University, Wuhan 430062, P. R. China}

\email{jliao@hubu.edu.cn(J. Liao), ykliu3612@qq.com(Y. Liu)}

\begin{document}

\begin{abstract}
In this paper, we prove that for $p\geq 13$ and $5\leq n\leq p-3$, every reduction simple fusion system on the finite $p$-group $S(n,p)$ is block-exotic. Furthermore, we establish the block-exoticity of three infinite families of reduction simple exotic fusion systems.
\end{abstract}

\maketitle

%Keywords: Fusion systems; Blocks;  Exotic fusion systems
%\textbf{2000 Mathematics Subject Classification:} 20C20; 20C30; 20C33; 20C34

\section{Introduction}

Let $p$ be a prime and $S$ be a finite $p$-group. A fusion system on $S$ is a category whose objects are the subgroups of $S$ and whose morphisms are injective group homomorphisms satisfying certain axioms. A fusion system is saturated if it satisfies two additional axioms, namely the Sylow axiom and the extension axiom. Throughout this paper, all fusion systems are assumed to be saturated unless otherwise stated.

A fusion system $\mathcal{F}$ on $S$ is called realizable if there exists a finite group $G$ with $S \in \Syl_{p}(G)$ such that $\mathcal{F} = \mathcal{F}_{S}(G)$, where $\mathcal{F}_{S}(G)$ is the fusion system on $S$ whose morphisms are given by conjugation by elements of $G$; otherwise, $\mathcal{F}$ is called exotic. 

Fix an algebraically closed field $k$ of characteristic $p$. In block theory, a fusion system can also be constructed from a block $b$ of the group algebra $kG$. More precisely, given a finite group $G$, a block $b$ of $kG$, and a maximal $b$-Brauer pair $(S, e_S)$, one obtains a fusion system $ \mathcal{F}_{(S, e_S)}(G,b)$ on the defect group $S$, whose morphisms are given by conjugation by elements of $G$ on the $b$-Brauer pairs contained in $(S, e_S)$. A fusion system $ \mathcal{F}$ on $S$ is called block-realizable if $ \mathcal{F} =  \mathcal{F}_{(S, e_S)}(G,b)$, in this case, $b$ is called an $ \mathcal{F}$-block; otherwise, $ \mathcal{F}$ is called block-exotic.

Recall that Brauer's third main theorem (see \cite[Part IV, Theorem 5.9]{AKO}) states that:
if $G$ is a finite group and $b$ is the principal block of $kG$, then for any maximal $b$-Brauer pair $(S, e_{S})$, the subgroup $S$ is a Sylow $p$-subgroup of $G$. Moreover, the associated fusion system $\mathcal{F}_{(S,e_{S})}(G,b)$ coincides with the group fusion system $\mathcal{F}_{S}(G)$. In particular, if a fusion system $\mathcal{F}$ is realizable, then it is block-realizable. That is to say, if $\mathcal{F}$ is block-exotic, then $\mathcal{F}$ is exotic. However, the converse remains an open problem:

\begin{conjecture}\cite[Part IV, 7.1]{AKO}
Let $\mathcal{F}$  be a fusion system on a finite $p$-group $S$.
If $\mathcal{F}$ is exotic, then $\mathcal{F}$ is block-exotic.
\end{conjecture}

Several families of fusion systems supporting this conjecture have already been identified.
The first family consists of the Solomon systems, defined on a Sylow $2$-subgroup of $\operatorname{Spin}(7,q)$. These are conjectured to be the only exotic fusion systems on $2$-groups. Block-exoticity for $q = 3$ was proved by Kessar in \cite{Ke}, and generalized to all odd prime powers $q$ in \cite[Theorem~9.34]{Cr}. The second example consists of the Ruiz–Viruel systems, defined on the extra-special $7$-group of order $7^3$; block-exoticity for these systems was established by Kessar and Stancu in \cite{KS}. 
Serwene established further families of block-exotic fusion systems in \cite{Se}, \cite{Se3} and \cite{Se2}. Specifically, in \cite{Se} it is proved that each exotic fusion system $\mathcal{F}$ on a Sylow $p$-subgroup of $G_2(p)$ for an odd prime $p$ with $O_p(\mathcal{F}) = 1$ is block-exotic. In \cite{Se3}, the exotic fusion system described by Grazian on a subgroup of the Monster group is shown to be block-exotic, implying that exotic and block-exotic fusion systems coincide for all $p$-groups of sectional rank $3$ with $p \geq 5$. In \cite{Se2} it is proved that exotic and block-exotic fusion systems coincide for all fusion systems on exceptional $p$-groups of maximal nilpotency class with $p \geq 5$, as well as for two exotic fusion systems on 3-rank two $3$-groups of maximal nilpotency class discovered by Diaz, Ruiz and Viruel in \cite{DRV}.

The present paper aims to provide further evidence for Conjecture 1.1. Our main results concern reduction simple fusion systems, which are those whose underlying $p$-group has no nontrivial proper strongly closed subgroup.
\begin{theorem}
For $p \geq 13$ and $5 \leq n \leq p-3$, any reduction simple fusion system on $S(n,p)$ is block-exotic.
\end{theorem}
 
 As a consequence, we prove that three infinite families of reduction simple exotic fusion systems are block-exotic. In \cite{Cl}, Clelland constructed an infinite family of fusion systems $\mathcal{E}(n,p)$ on the $p$-group $S(n,p)$, where $p$ is an odd prime and $2 \leq n \leq p-1$. In \cite{CP}, Clelland and Parker constructed two infinite families of fusion systems $\mathcal{F}(r,n,k,X)$ for $X \in \{R, Q\}$, where $k$ is an arbitrary finite field.

\begin{theorem}
For $p \geq 13$ and $5 \leq n \leq p-3$, the fusion system $\mathcal{E}(n,p)$ is block-exotic.
\end{theorem}
\begin{theorem}
For $p \geq 13$ and $5 \leq n \leq p-3$, the fusion systems $\mathcal{F}(r,n,\mathbb{F}_{p},R)$ and $\mathcal{F}(r,n,\mathbb{F}_{p},Q)$ are block-exotic.
\end{theorem}

The paper is divided into four sections. In Section 2, we frist introduce the $p$-group $S(n,p)$, and then describe the fusion systems $\mathcal{E}(n,p)$ and $\mathcal{F}(r,n,\mathbb{F}_{p},X)$ for $X \in \{R, Q\}$, all of which are defined on $S(n,p)$. In Section 3, we gather some preliminary facts, some of which are possibly well known, in preparation for the proof of Theorem 1.2. In Section 4, we prove Theorem 1.2 using the classification of finite simple groups. As a consequence, we prove that for $p\geq 13$ and $5 \leq n \leq p-3$, the exotic fusion systems $\mathcal{E}(n,p)$ and $\mathcal{F}(r,n,\mathbb{F}_{p},X)$ for $X \in \{R, Q\}$ are block-exotic.

\section{ The systems $\mathcal{E}(n,p)$ and $\mathcal{F}(r,n,\mathbb{F}_{p},X)$}

We briefly introduce the systems $\mathcal{E}(n,p)$ and $\mathcal{F}(r,n,\mathbb{F}_{p},X)$ for $X \in \{R, Q\}$. All of these families are defined on the same $p$-group $S(n,p)$. In fact the $p$-groups $S(n,p)$ are Sylow $p$-subgroups of the semidirect products that arise naturally from considering the so-called basic irreducible  $\mathbb{F}_{p}\GL_{2}(p)$-modules as described by Brauer and Nesbitt in \cite{BN}.

Let $\mathbb{F}_{p}$ be the field with $p$ elements, where $p$ is an odd prime, and let $\mathbb{F}_{p}[x,y]$ be the polynomial ring in two indeterminates with coefficients in $\mathbb{F}_{p}$. The group $\GL_{2}(p)$ acts on $\mathbb{F}_{p}[x,y]$ by linear transformation: let $X = \left(\begin{smallmatrix}
\alpha & \beta \\
\gamma & \delta
\end{smallmatrix}\right) \in \GL_{2}(p)$, define an action of $X$  on $\mathbb{F}_{p}[x,y]$ by setting $x \mapsto (\alpha x + \beta y)$ and $y \mapsto (\gamma x + \delta y)$, and then extend linearly to all of $\mathbb{F}_{p}[x,y]$. For an integer $n$ with $2 \leq n \leq p-1$, denote by $A(n,p)$ the $n+1$-dimensional subspace of $\mathbb{F}_{p}[x,y]$ consisting of all homogeneous polynomials of degree $n$. The multiplicative group $\mathbb{F}^{*}_{p}$ acts on $A(n,p)$ by scalar multiplication: for $\lambda \in \mathbb{F}^{*}_{p}$ and $f \in A(n,p)$, define $f^{\lambda} = \lambda f$. Set $\Gamma = \GL_{2}(p) \times \mathbb{F}^{*}_{p}$. Then $A(n,p)$ becomes a $\mathbb{F}_{p} \Gamma$-module via $f^{(X,\lambda)} = \lambda(f^{X})$. Now form the semidirect product $\mathcal{G}(n,p) = A(n,p) \rtimes \Gamma$, where $A(n,p)$ is viewed as an elementary abelian $p$-group. It is easy to check that $|\mathcal{G}(n,p)|$ = $|A(n,p)||\GL_{2}(p)||\mathbb{F}^{*}_{p}| = p^{n+2}(p-1)^{3}(p+1)$. Define $S(n,p) = A(n,p) \rtimes \langle(\left(\begin{smallmatrix}
1 & 0 \\
1 & 1
\end{smallmatrix}\right),1) \rangle$.
Then $|S(n,p)| = p^{n+2}$ and $S(n,p)$ is a Sylow $p$-subgroup of $\mathcal{G}(n,p)$. 
For more details on the construction of $S(n,p)$, see \cite[Section 5.1]{Cl}.

{\bf The system $\mathcal{E}(n,p)$}. In \cite[after Proposition 5.2.4]{Cl}, Clelland constructed an infinite family of fusion systems $\mathcal{E}(n,p)$ on $p$-group $S(n,p)$, where $p$ is an odd prime and $2 \leq n \leq p-1$. Let $T = \langle x^{n}\rangle \rtimes \langle(\left(\begin{smallmatrix}
1 & 0 \\
1 & 1 
\end{smallmatrix}\right),1)  \rangle \leq S(n,p)$. Define $\mathcal{E}(n,p)$ to be the fusion system generated by $\mathcal{F}_{S}(\mathcal{G})$ and $\Aut(T)$, i.e., 
\[\mathcal{E}(n,p) = \langle \mathcal{F}_{S}(\mathcal{G}), \Aut(T)\rangle\] 
where $S=S(n,p)$ and $\mathcal{G}=\mathcal{G}(n,p)$.
For $p\geq 13$ and $5 \leq n \leq p-1$, Clelland proved that $\mathcal{E}(n,p)$ is exotic in  \cite[Chapter 6]{Cl} .

{\bf The systems $\mathcal{F}(r,n,\mathbb{F}_{p},X)$}. 
In \cite[Theorem 4.9]{CP}, Clelland and Parker constructed two infinite families of fusion systems $\mathcal{F}(r,n,k,X)$ for $X \in \{R, Q\}$, where $k$ is an arbitrary finite field. We only consider the special case $k = \mathbb{F}_{p}$.  In this case, both families are defined on the same $p$-group $S(n,p)$ introduced above. Define the following matrix groups:
\begin{center}
$P_{R} = \left\{ Y = \left(\begin{smallmatrix}
x_{1} & x_{2} & x_{3} \\
x_{4} & x_{5} & x_{6} \\
0 & 0 & x_{7} 
\end{smallmatrix}\right)| Y \in \GL_{3}(p) \right\}$
\end{center}
and\\
\begin{center}
$P_{Q} = \left \{ Y = \left(\begin{smallmatrix}
x_{1} & x_{2} & x_{3} & x_{4} \\
0 & x_{5} & x_{6} & x_{7} \\
0 & x_{8} & x_{9} & x_{10} \\
0 & 0 & 0 & x_{11} 
\end{smallmatrix}\right)| Y \in \GSp_{4}(p)\right\}$.
\end{center} 
Let $R$ and $Q$ be the $p$-subgroups of $S(n,p)$ defined in \cite{CP} (see the discussion preceding Lemma 4.4). For $X \in \{R,Q\}$, let $F_{X} = F(n,\mathbb{F}_{p},X) = \mathcal{G}(n,p)*_{N_{\mathcal{G}(n,p)}(X)}P_{X}$ be the free amalgamated product, where the specific monomorphisms $\psi_{R}$, $\psi_{Q}$ for the free amalgamated products are given in \cite{CP}. Set $L = N_{O^{p'}(P_{X})}(N_{S}(X))O^{p'}(\mathcal{G})$ and $L_{X} = (L \cap P_{X})O^{p'}(P_{X})$.
Define $F^{*}_{X} = F^{*}(n,\mathbb{F}_{p},X) = L*_{L \cap L_{X}}L_{X}$, which can be view as a subgroup of $F_{X}$. Let $r$ be a divisor of $p-1$. Denote by $F(r,n,\mathbb{F}_{p},X)$ the unique subgroup of $F_{X}$ containing $F^{*}_{X}$ and of index $r$. Finally, define the fusion system $\mathcal{F}(r,n,\mathbb{F}_{p},X) = \mathcal{F}_{S}(F(r,n,\mathbb{F}_{p},X))$. The parameters $r$ and $n$ are subject to the following conditions.
\begin{enumerate}
\item For $X = R$: $n \geq 2$, $p$ odd, and $r \mid (n+2, p-1)$.
\item For $X = Q$: $n \geq 3$, $p \geq 5$, and $r \mid (n, p-1)$.
\end{enumerate}
Under these conditions, \cite[Theorem 5.1-5.2]{CP} shows that the fusion systems $\mathcal{F}(r,n,\mathbb{F}_{p},X)$ are exotic, except possibly when $n = 2$ and $p = 3$ or 5 for $X = R$.

In the remainder of this paper, we restrict to the case $p \geq 13$ and $5 \leq n \leq p-3$. Under these hypothesis, all three families $\mathcal{E}(n,p)$, $\mathcal{F}(r,n,\mathbb{F}_{p},R)$, and $\mathcal{F}(r,n,\mathbb{F}_{p},Q)$ are exotic.

We conclude this section with a key property of these fusion systems. Recall that a fusion system $\mathcal{F}$ on a finite $p$-group $S$ is called reduction simple if $S$ contains no nontrivial proper strongly $\mathcal{F}$-closed subgroup.

\begin{lemma}
Let $S = S(n,p)$, then
\begin{enumerate}
\item $|Z(S)| = p$ and  $Z(S) \subseteq \left[S, S \right]$.
\item $\mathcal{E}(n,p)$ is a reduction simple fusion system on $S$.
\item $\mathcal{F}(r,n,\mathbb{F}_{p},X)$ is a reduction simple fusion system on $S$.
\end{enumerate}
\begin{proof}
This follows from \cite[Lemma 5.1.1]{Cl}, \cite[Lemma 6.1.3]{Cl}, and \\
\cite[Lemma 5.5]{CP}.
\end{proof}
\end{lemma}

\section{ Preliminaries}
Let $k$ be an algebraically closed field of characteristic $p$. 

\begin{notation*}
Let $G$ be a finite group and let $kG$ be its group algebra over $k$. A $p$-block of $G$ is a block of $kG$. For such a block $b$, we denote its defect group by $\delta(b)$. Let $\nu$ be the $p$-adic valuation normalized by $\nu(p)=1$; for a finite group $H$, set $\nu(H)=\nu(|H|)$. For a real number $x$, let $[x]$ denote the greatest integer not exceeding $x$.
\end{notation*}

We start with Serwene's reduction theorem, which reduces the problem to the quasisimple case.
\begin{theorem}
\cite[Theorem 3.5]{Se} Let $S$ be a non-abelian $p$-group such that $Z(S)$ is cyclic, and let $\mathcal{F}$ be a reduction simple fusion system on $S$. If $\mathcal{F}$ is block-realizable, then there exists a fusion system $\mathcal{F}_{0}$ on $S$ and a quasisimple group $L$ with an $\mathcal{F}_{0}$-block, where $\mathcal{O}_{p}(\mathcal{F}_{0}) = 1$. 
\end{theorem}

The following results are useful for proving Theorem 1.2.

\begin{theorem} \cite[Chapter 5, Theorem 8.8]{NT}
Let $H$ be a normal $p'$-subgroup of a finite group $G$, and set $\overline{G} = G/H$. Denote by $\mu_{H}$ the natural $k$-algebra homomorphism 
\begin{center}
$\mu_{H}: kG \rightarrow k \overline{G} \qquad (\sum_{x} \alpha_{x}x \mapsto \sum_{x} \alpha_{x} \overline{x})$
\end{center}
if $ \mu_{H}(b) \neq 0$, where $b$ is a $p$-block of group $G$. Then $\mu_{H}(b)$ is precisely the $p$-block, say $\overline{b}$, of $\overline{G}$ and it holds that 
\begin{center}
\qquad $\delta(\overline{b}) = {}_G {\overline{\delta(b)}}$
\end{center}
i.e., the defect group $\delta(\overline{b})$ is $G$-conjugate to the quotient $\overline{\delta(b)} $ of the defect group $\delta(b)$.

\end{theorem}

\begin{remark}
The notation and conditions are the same as in Theorem 3.2, since $1_{G} = \sum b$, where the sum runs over all $p$-blocks $b$ of $G$, we obtain $1_{\overline{G}} = \sum \mu_{H}(b)$. By Theorem 3.2, each term $\mu_{H}(b)$ is either 0 or a $p$-block of $\overline{G}$. Therefore, by \cite[Part IV, Proposition 1.9 (a)]{AKO}, for any $p$-block $\tilde{b}$ of $\overline{G}$, there exists a unique $p$-block $b$ of $G$ such that $\mu_{H}(b) = \tilde{b}$. By Theorem 3.2, and using the fact that $H$ is a $p'$-group, we obtain $\delta(\tilde{b}) = {}_G {\overline{\delta(b)}} \cong \delta(b)H/H \cong \delta(b)$. In other words, the defect group of each $p$-block of $\overline{G}$ is isomorphic to the defect group of the corresponding $p$-block of $G$.
\end{remark}

\begin{definition}\cite[Definition 5.1.1]{GLS}
Let $K$ be a finite quasisimple group. A covering of $K$ is a finite quasisimple group $L$ together with a surjective homomorphism $L \rightarrow K$. This covering is universal if and only if for every covering $M \rightarrow K$ ($M$ quasisimple), there is a unique covering $L\rightarrow M$ such that the following diagram commutes:

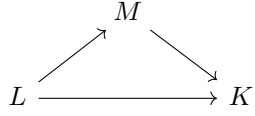
\[
\begin{tikzcd}
& M \arrow{dr}{} \\
L \arrow{ur}{} \arrow{rr}{} & & K
\end{tikzcd}
\]

If $L \rightarrow K$ is a (universal) covering, we refer to $L$ as a (universal) covering group of $K$. Furthermore, the kernel $Z$ of the covering is a normal subgroup of $L$ with $L/Z \cong K$. Since $L$ is quasisimple, $Z(L)$ is the unique maximal proper normal subgroup of $L$, so $Z \leq Z(L)$.

\end{definition}

\begin{definition}\cite[Definition 5.1.6]{GLS}
The Schur multiplier of a quasisimple group $K$, denoted by $M(K)$, is the kernel of a universal covering $L \rightarrow K$ of $K$.
\end{definition}

\begin{proposition} \cite[Corollary 5.1.5 (a)]{GLS}
If $K$ is a simple group and $L$ a universal covering group of $K$, then any quasisimple group $G$ with $G/Z(G) \cong K$ is isomorphic to $L/Z$ for some subgroup $Z \leq Z(L)$ (note that when $K$ is a simple group, $Z(L)=M(K)$).
\end{proposition}

\begin{remark}
The notation and conditions are the same as in Proposition 3.6, we have $|G| = |L/Z| = |L/Z(L)|  |Z(L)/Z| = |K|  |M(K)|/|Z| = n|K|$, where $n = |M(K)|/|Z|$ is a divisor of $|M(K)|$.
\end{remark}
 
For alternating groups and symmetric groups, we have the following results.

\begin{proposition}\cite[Theorem 6.2.39]{JK}
Let $p$ be a prime. Then the defect groups of the $p$-blocks of symmetric group $\mathfrak{S}_{n}$ are conjugates of subgroups of form 
\begin{center}
$D^{w}= \psi[C_{p} \wr (\mathfrak{S}_{w})_{p}]$ 
\end{center}
for some $w\in \{1, 2, \dots, [n/p] \}$, where $(\mathfrak{S}_{w})_{p}$ denotes a Sylow $p$-subgroup of symmetric group $\mathfrak{S}_{w}$ and $\psi: C_{p} \wr (\mathfrak{S}_{w})_{p} \rightarrow \mathfrak{S}_{pw}$ is a faithful representation (see \cite[4.1.18]{JK}).
\end{proposition}

\begin{lemma}
Let $p$ be a prime. Then the defect groups of the $p$-blocks of symmetric group $\mathfrak{S}_{n}$ are isomorphic to a Sylow $p$-subgroup of some symmetric group $\mathfrak{S}_{l}$.
\begin{proof}
By Proposition 3.8, it suffices to show that the order of $C_{p} \wr (\mathfrak{S}_{w})_{p}$ equals the order of a Sylow $p$-subgroup of $\mathfrak{S}_{pw}$. Note that $\nu(C_{p} \wr (\mathfrak{S}_{w})_{p}) = \nu(|C_{p}|^{w} |(\mathfrak{S}_{w})_{p}|) = w+\nu(w!) = w+ \sum_{i=1}^{\infty}[w/p^{i}] = w+ \sum_{i=2}^{\infty}[w/p^{i-1}] = \sum_{i=1}^{\infty}[pw/p^{i}] = \nu((pw)!) = \nu(\mathfrak{S}_{pw})$. Setting $l = pw$ completes the proof.
\end{proof}
\end{lemma}

\begin{lemma}
Let $p$ be an odd prime. Then the defect groups of the $p$-blocks of a double cover $2.\mathfrak{S}_{n}$ of $\mathfrak{S}_{n}$ are isomorphic to a Sylow $p$-subgroup of some symmetric group $\mathfrak{S}_{l}$.
\begin{proof}

Let  $Z = \{1, z\}$ be the kernel of the covering $2.\mathfrak{S}_{n}\rightarrow \mathfrak{S}_{n}$, where $z$ is an element of order 2. From the discussion preceding \cite[Theorem 1.2]{Ol}, since $p$ is odd, a $p$-block of $2.\mathfrak{S}_{n}$ either contains only spin characters or contains no spin characters at all. Here spin characters are the faithful irreducilbe characters, which satisfy $\chi(z) = -\chi(1)$. Non-spin irreducilbe characters satisfy $\chi(z) = \chi(1)$; see \cite[1.4]{Cab} for details.

If the block contains only spin characters, by \cite[Theorem 1.3]{Ol}, its defect group is isomorphic to a Sylow $p$-subgroup of some double cover $2.\mathfrak{S}_{l}$. Moreover, since $p$ is odd, a Sylow $p$-subgroup of $2.\mathfrak{S}_{l}$ is isomorphic to a Sylow $p$-subgroup of $\mathfrak{S}_{l}$. Hence, the defect group is isomorphic to a Sylow $p$-subgroup of the symmetric group $\mathfrak{S}_{l}$. 

If instead the block contains no spin characters, then $Z \leq \ker \chi$. It follows from \cite[Chapter 5, Lemma 8.6 (i)]{NT} and Theorem 3.2 that the defect group of this $p$-block is isomorphic to a defect group of a $p$-block of $\mathfrak{S}_{n}$. Hence, by Lemma 3.9, this defect group is isomorphic to a Sylow $p$-subgroup of some symmetric group $\mathfrak{S}_{l}$.
\end{proof}
\end{lemma}

The following results concern groups of Lie type over fields of characteristic $q \neq p$.

\begin{proposition}
\cite[Proposition 4.3]{Se}
Let $G$ be a quasisimple finite group, and denote the quotient $G/Z(G)$ by $\overline{G}$. Suppose $\overline{G} = G(q)$ is a finite group of Lie type, and let $p$ be a prime number $\geq 7$, $(p,q) = 1$. Let $D$ be a $p$-group such that $Z(D)$ is cyclic of order $p$ and $Z(D) \subseteq \left[D, D \right]$. If $D$ is a defect group of a block of $G$, then there are $n, k \in \mathbb{N}$ and a finite group $H$ with 
\begin{center}
$\SL_{n}(q^{k}) \leq H \leq \GL_{n}(q^{k})$ \quad (or $\SU_{n}(q^{k}) \leq H \leq \GU_{n}(q^{k})$) 
\end{center}
such that there is a block $c$ of $H$ with non-abelian defect group $D'$ such that $D'/Z$ is of order $|D/Z(D)|$ for some $Z \leq D'\cap Z(H)$.
\end{proposition}

\begin{proposition} \cite[Theorem 3C]{FS} 
Let $p$ be an odd prime with $(p,q) = 1$, let $G = \GL_{n}(q^{k})$ or $\GU_{n}(q^{k})$, and let $R$ be a defect group of a $p$-block of $G$. Then $R$ is conjugate to 
\begin{equation*}
\begin{pmatrix}
R_{0} &        &        &        \\
       & R_{1} &        &        \\
       &        & \ddots &        \\
       &        &        & R_{t}  
\end{pmatrix}
\end{equation*}
where $R_{0}$ is an identity matrix of non-negative degree, and for $i \geq 1$, there exist integers $m_{i}$, $ \alpha_{i}$, $ \beta_{i}$ with $m_{i} \geq 1$, $\alpha_{i} \geq 0$, $\beta_{i} \geq 0$ such that $R_{i} = R^{m_{i}, \alpha_{i}, \beta_{i}} = R^{m_{i}, \alpha_{i}} \wr X_{\beta_{i}}$. Here $R^{m_{i}, \alpha_{i}}$ is a cyclic group of order $p^{a + \alpha_{i}}$, where $a = \nu(q^{ke}-1)$, and $e$ is the order of $q^{k}$ modulo $p$. Finally, $X_{\beta_{i}}$ is a Sylow $p$-subgroup of the group of permutation matrices of degree $p^{\beta_{i}}$. 
\end{proposition}

\begin{lemma}
Let $p$ be an odd prime with $(p,q) = 1$, and let $G = \GL_{n}(q^{k})$ or $\GU_{n}(q^{k})$. Then any non-abelian defect group of a $p$-block of $G$ has order at least $p^{pa+1}$, where $a$ is the same as defined in Proposition 3.12.
\begin{proof}
By Proposition 3.12, $\nu(R^{m, \alpha, \beta}) = \nu(R^{m, \alpha} \wr X_{\beta}) = \nu(|R^{m, \alpha}|^{p^{\beta}}  |X_{\beta}|) = (a + \alpha)p^{\beta}+\nu(p^{\beta}!)$. To obtain a non-abelian defect group $R$ of minimal order, we must take $\alpha = 0$ and $\beta = 1$ (note that if $\beta = 0$, then $R^{m, \alpha, \beta} = R^{m, \alpha}\wr X_{\beta}$ is abelian, so we need $\beta = 1$). In this case, $\nu(R) = \nu(R^{m, 0, 1}) = (a+0)p^{1}+\nu(p!) = pa+1$. 
\end{proof}
\end{lemma}

\begin{proposition}
If $G$ is as defined in Proposition 3.11, then $G$ has no $p$-blocks with defect groups isomorphic to $S(n,p)$, where $p \geq 13$ and $5 \leq n \leq p-3$.
\begin{proof}
Fix integers $p \geq 13$ and $5 \leq n \leq p-3$, and set $S = S(n,p)$. Note that $|S| = p^{n+2}$ and $|Z(S)| = p$. Since $p \geq 13$, apply Proposition 3.11 with $D=S$. Then there exist $m, k \in \mathbb{N}$, a finite group $H$, and a subgroup $D'$ as in Proposition 3.11. Assume first $\SL_{m}(q^{k}) \leq H \leq \GL_{m}(q^{k})$ and let $a=\nu(q^{ke}-1)$ (note that $a \geq 1$). Then  we have
\begin{align*}
|D'| &= |S/Z(S)| |Z| \leq p^{n+1}|Z(H)|_{p} \leq p^{n+1}|C_{\GL_{m}(q^{k})}(\SL_{m}(q^{k}))|_{p} 
\\&= p^{n+1}|q^{k}-1|_{p} \leq p^{n+1}|q^{ke}-1|_{p} = p^{n+1+a}.
\end{align*}
Now there is a block of $k\GL_{m}(q^{k})$ covering $c$, with non-abelian defect group of order at most $p^{n+1+2a}$. Indeed, note that $\SL_{m}(q^{k}) \trianglelefteq H \trianglelefteq \GL_{m}(q^{k})$. If we assume that the defect group is $D''$, then by \cite[Chapter 5, Theorem 5.16 (ii)]{NT}, we have $|D'| = |D'' \cap H| = |D''|  |H|/|D''H|$. It follows that 
\begin{align*}
|D''| &= |D'|  |D''H|/|H| \leq |D'|  | \GL_{m}(q^{k})/ \SL_{m}(q^{k})|_{p} \\&= |D'||q^{k}-1|_{p} \leq |D'|  |q^{ke}-1|_{p} \leq p^{n+1+2a}. 
\end{align*}
However, by Lemma 3.13, any non-abelian defect group of $\GL_{m}(q^{k})$ has order at least $p^{pa+1}$. Thus $p^{pa+1} \leq p^{n+1+2a}$, which contradicts our earlier assumption that $n \leq p-3$.

 A similar argument applies to the case $\SU_{m}(q^{k}) \leq H \leq \GU_{m}(q^{k})$. In fact, assume $\SU_{m}(q^{k}) \leq H \leq \GU_{m}(q^{k})$ (note that $q^{k}$ must be a square number in this case), then $C_{\GU_{m}(q^{k})}(\SU_{m}(q^{k})) = Z(\GU_{m}(q^{k})) = \{ \lambda I_{m} | \lambda \in {\mathbb{F}}^{\times}_{q^{k}}, ({\lambda I_{m}})^{*}(\lambda I_{m}) = I_{m} \} = \{\lambda I_{m} | \lambda \in {\mathbb{F}}^{\times}_{q^{k}}, \lambda^{\sqrt{q^{k}}+1} = 1 \}$. Let $a=\nu(q^{ke}-1)$ (note that $a \geq 1$). Then  we have
\begin{align*}
|D'| &= |S/Z(S)| |Z| \leq p^{n+1}|Z(H)|_{p} \leq p^{n+1}|C_{\GU_{m}(q^{k})}(\SU_{m}(q^{k}))|_{p} 
\\&= p^{n+1}|\sqrt{q^{k}}+1|_{p} \leq p^{n+1}|q^{k}-1|_{p} \leq p^{n+1}|q^{ke}-1|_{p} = p^{n+1+a}.
\end{align*}
Now there is a block of $k\GU_{m}(q^{k})$ covering $c$, with non-abelian defect group of order at most $p^{n+1+2a}$. Indeed, note that $\SU_{m}(q^{k}) \trianglelefteq H \trianglelefteq \GU_{m}(q^{k})$. If we assume that the defect group is $D''$, then by \cite[Chapter 5, Theorem 5.16 (ii)]{NT}, we have $|D'| = |D'' \cap H| = |D''|  |H|/|D''H|$. It follows that 
\begin{align*}
|D''| &= |D'|  |D''H|/|H| \leq |D'|  | \GU_{m}(q^{k})/ \SU_{m}(q^{k})|_{p} \\&= |D'||\sqrt{q^{k}}+1|_{p} \leq |D'|  |q^{k}-1|_{p} \leq |D'|  |q^{ke}-1|_{p} \leq p^{n+1+2a}. 
\end{align*}
However, by Lemma 3.13, any non-abelian defect group of $\GU_{m}(q^{k})$ has order at least $p^{pa+1}$. Thus $p^{pa+1} \leq p^{n+1+2a}$, which contradicts our earlier assumption that $n \leq p-3$.
\end{proof}
\end{proposition}

The following results concern groups of Lie type over fields of characteristic $p$.

\begin{lemma} \cite[Lemma 5.1]{Ke}
Let $p$ be a prime and let $X$ be a finite group having a Tits system $(B, N, R)$. Set $T = B \cap N$ and $W = N/T$. Suppose that the following conditions hold:
\begin{enumerate}[label=(\roman*)]
\item The $BN$ pair for $X$ is split of characteristic $p$, that is $B = U \rtimes T$, where $U$ is a normal $p$-subgroup of $B$ and $T$ is a $p'$-group.
\item The $BN$ pair for $X$ is strongly split, that is for any $J \subset R$, and any $n \in N$ such that the image of $n$ in $W$ is the longest element $\omega_{J}$ of $\langle J \rangle$, $U \cap U^{n}$ is a normal subgroup of $U$.
\item If $Q$ is any $p$-subgroup of $X$ satisfying $O_{p}(N_{X}(Q)) = Q$, then there is an $x\in X$ and an $I \subset R$ such that $N_{X}(^{x}Q) = P_{I} = \bigcup_{w \in \langle I \rangle }BwB$.
\item $W$ is of irreducible type, that is, if we denote by $\Delta = \{\alpha_{r}|r \in R\}$ a basis of $\mathbb{R}^{|R|}$ such that the element $r$ of $R$ acts as the reflection in the hyperplane perpendicular to $\alpha_{r}$ in the natural representation of $W$ as a reflection group in $\mathbb{R}^{|R|}$ (see \cite[Section 2.2]{Ca}), then there is no partition of $\Delta$ into non-empty orthogonal subsets. 
\end{enumerate}
Then every $p$-subgroup of $X$ which is of defect type in $X$ is either the identity group or a Sylow $p$-subgroup of $X$.
\end{lemma}

\begin{lemma}
Let $p \geq 5$ be a prime, and let $G$ be a quasisimple finite group such that $G/Z(G)$ is a simple group of Lie type over a field of characteristic $p$. If $D$ is a defect group of a $p$-block of $G$ and $D \neq 1$, then $D$ is a Sylow $p$-subgroup of $G$. Moreover, $Z(G)$ is a $p'$-group and $D$ is isomorphic to a Sylow $p$-subgroup of $G/Z(G)$.
\begin{proof}
Suppose that $\overline{G} = G/Z(G)$ is a simple group of Lie type over a field of characteristic $p$. Then the exceptional part of the Schur multiplier $M(\overline{G})$ of $\overline{G}$ is trivial \cite[Table 6.1.3]{GLS}, and in this case, $M(\overline{G})$ is a $p'$-group (see \cite[Table 6.1.2 and Theorem 6.1.4]{GLS}). Thus, by \cite[Remark 24.19]{MT}, there is an extension $\widetilde{G}$ of $\overline{G}$
\[
1 \longrightarrow M(\overline{G}) \longrightarrow \widetilde{G} \longrightarrow \overline{G} \longrightarrow 1
\]
such that $\widetilde{G} = \textbf{G}^{F}$, where $\textbf{G}$ is a simply connected simple algebraic group and $F$ is a Frobenius endomorphism of $\textbf{G}$. In fact, $\widetilde{G}$ is a universal covering group of $\overline{G}$. Moreover, by Proposition 3.6, there is an extension  
\[
1 \longrightarrow Z \longrightarrow \widetilde{G} \longrightarrow G \longrightarrow 1
\]
where $Z$ satisfies $Z \leq Z(\widetilde{G}) = M(\overline{G})$ and is a $p'$-group. The group $\widetilde{G} = \textbf{G}^{F}$ has a split $BN$ pair of characteristic $p$ (see \cite[Section 1.18]{Ca}), satisfying condition (ii), condition (iii), and condition (iv) of Lemma 3.15 (see the Remark 3.17 below). Therefore, the defect groups of a $p$-block of $\widetilde{G}$ are either trival or Sylow $p$-subgroups of $\widetilde{G}$, and then the same is true for $G$ (see Remark 3.3). Hence, $D$ is a Sylow $p$-subgroup of $G$. By Remark 3.7, we have $|Z(G)|\bigm| |M(\overline{G})|$, so $Z(G)$ is a $p'$-group. Consequently, $D$ is isomorphic to a Sylow $p$-subgroup of $G/Z(G)$.
\end{proof}
\end{lemma}

\begin{remark}
The group $\widetilde{G} = \textbf{G}^{F}$ in Lemma 3.16 satisfies conditions (ii) (iii) and (iv) of Lemma 3.15. 

\begin{proof}
We will verify each condition in turn.

Condition (ii):
From the paragraph immediately preceding \cite[Proposition 2.6.4]{Ca}, it follows that $\textbf{G}^{F}$ is an algebraic group with a split $BN$-pair which satisfies the commutator relations. By \cite[Proposition 2.6.4]{Ca}, $U \cap U^{n} \trianglelefteq  P_{J} = \bigcup_{w \in \langle J \rangle }BwB$, note that $U \subseteq B \subseteq P_{J}$, so $U \cap U^{n} \trianglelefteq U$.

Condition (iii):
Let $X = \textbf{G}^{F} \in \mathcal{L}ie(p)$ (see \cite[Theorem 2.2.6 (f)]{GLS}, \cite[Defination 2.2.2]{GLS} and \cite[Defination 1.10.5]{GLS}). By \cite[Corollary 3.1.5]{GLS}, $N_{X}(Q)$ is a parabolic subgroup of $X$, by \cite[Corollary 1.13.2 (b)]{GLS}, there exists $x\in X$ and an $I \subset R$ such that $N_{X}(^{x}Q) = {^{x}N_{X}(Q)} = P_{I} = \bigcup_{w \in \langle I \rangle }BwB$.

Condition (iv) : We present the proof in steps.

Step 1: determine the Tits system.
In \cite[Theorem 2.3.4 (b)]{GLS}, we take $\bar{K} = \textbf{G}$, $\sigma = F$, and $K^{*} = C_{\textbf{G}}(F) = \textbf{G}^{F} = \widetilde{G}$, then we obtain a Tits system $(B, N, R) = (C_{\textbf{B}}(F), C_{\textbf{N}}(F), S)$ for $\widetilde{G}$, where $\textbf{B}$ and $\textbf{N}$ denote the Borel subgroup and the normalizer of a maximal torus in $\textbf{G}$, respectively, and the set $S$ of generating reflections provided by \cite[Theorem 2.3.4 (b)]{GLS}.

Step 2: identify the Weyl group.
By \cite[Theorem 2.3.4 (a)]{GLS}, $W_{\widetilde{G}} = N/(B \cap N) \cong \widetilde{W} := C_{W_{\textbf{G}}}(\tau)$, where $W_{\textbf{G}}$ denotes the Weyl group of the algebraic group $\textbf{G}$ and $\tau$ is the isometry of the root system induced by $F$ (see \cite[Definition 2.3.1]{GLS}).

Step 3: identify the root system and its basis.
By \cite[Definition 2.3.1]{GLS} and \cite[Proposition 2.3.2 (a)]{GLS}, $\widetilde{W}$ is the Weyl group of the twisted root system $\widetilde{\Sigma}$, acting faithfully on $\widetilde{V} = C_{V}(\tau)$. By \cite[Proposition 2.3.2 (b)]{GLS}, the twisted root system $\widetilde{\Sigma}$ possesses a basis; in the $BN$-pair setting of \cite[Theorem 2.3.4 (b)]{GLS}, this basis is exactly the set $\hat{\Pi}$ used to define the generating reflections $S = \{ r_{\hat{\alpha}} | \hat{\alpha} \in \hat{\Pi} \}$. Consequently, the set $\Delta = \{\alpha_r | r \in R\}$ in Lemma 3.15 corresponds to $\hat{\Pi}$, which is the basis of the twisted root system $\widetilde{\Sigma}$.

Step 4: establish irreducibility.
Since $\mathbf{G}$ is a simple algebraic group, its original root system $\Sigma$ is irreducible (see \cite[Corollary 27.5]{Hum}). 
Under the hypothesis $p \geq 5$, the Suzuki-Ree cases in \cite[Proposition 2.3.4 (d) (3)]{GLS} do not occur. For the remaining cases in \cite[Proposition 2.3.4 (d) (1)-(2)]{GLS}, the classification shows that $\widetilde{\Sigma}$ is one of the irreducible root systems listed in the table $(e.g., C_n, B_n, G_2, F_4, etc.)$. Hence $\widetilde{\Sigma}$ is irreducible, and its basis $\hat{\Pi}$ admits no partition into non-empty orthogonal subsets.

Step 5: conclusion. 
Thus $\widetilde{W}$, as the Weyl group of the irreducible root system $\widetilde{\Sigma}$, has an irreducible reflection representation. Therefore, condition (iv) of Lemma 3.15 holds for $X = \widetilde{G}$. 
\end{proof}
\end{remark}

\begin{proposition}\cite[Proposition 6.3.1]{Cl}
Let $p$ be an odd prime, and let $q = p^{b}$ for some $b$. Let $K$ be a finite simple Chevalley group or a finite simple twisted Lie type group defined over a field of order q and let $U$ be a Sylow $p$-subgroup of $K$. Suppose that $U$ is non-abelian and let $X$ be an elementary abelian subgroup of $U$ with maximal possible rank in $U$. Then $rank(X) > 3$ implies that $|U : X| > p$.
\end{proposition}

\begin{lemma}
No finite simple group of Lie type defined over a field of characteristic $p$ has a Sylow $p$-subgroup isomorphic to $S(n,p)$, where $p \geq 13$ and $5 \leq n \leq p-3$.
\begin{proof}
The proof follows the same lines as that of \cite[Theorem 6.3.2]{Cl}.
Let $K$ be a finite simple Chevalley group or a simple twisted group of Lie type defined over a finite field of characteristic $p$. Suppose that $K$ has a Sylow $p$-subgroup isomorphic to $S(n, p)$. The group $S(n, p)$ is non-abelian and contains an elementary abelian subgroup $A(n,p)$ of rank $n + 1 \geq 6$ and index $p$. But by Proposition 3.18, $K$ has no such Sylow $p$-subgroup, a contradiction.
\end{proof}
\end{lemma}

\section{ Proof of Theorem 1.2}

\begin{proof}[Proof of Theorem 1.2]
Fix integers $p \geq 13$ and $5 \leq n \leq p-3$, and set $S = S(n,p)$. Let $\mathcal{F}$ be a reduction simple fusion system on $S$. By Theorem 3.1, if $\mathcal{F}$ is block-realizable, then there exists a fusion system $\mathcal{F}_{0}$ on $S$ and a quasisimple group $G$ with an $\mathcal{F}_{0}$-block. Clearly, $G$ is non-abelian. We will prove that this system $\mathcal{F}_{0}$ cannot be realized by any $p$-block of a quasisimple group $G$. This will yield a contradiction, thereby proving the theorem. 

First, assume that $G/Z(G)$ is an alternating group $\mathfrak{A}_{m}$ for some $m \geq 5$. By \cite[Theorem 5.2.3]{GLS}, The cases $m  =  6$ and $m  =  7$ are immediately excluded by the order of $S$. Then the group $G$ is either the alternating group $\mathfrak{A}_{m}$ or a double cover $2.\mathfrak{A}_{m}$. If $G = \mathfrak{A}_{m}$, note that $S$ is a defect group of a $p$-block of $\mathfrak{A}_{m}$ and that $\mathfrak{A}_{m} \trianglelefteq \mathfrak{S}_{m}$. Hence $S$ is isomorphic to a defect group of a $p$-block of the symmetric group $\mathfrak{S}_{m}$ (see \cite[Chapter 5, Theorem 5.16 (ii)]{NT}). By Lemma 3.9, $S$ is isomorphic to a Sylow $p$-subgroup of some symmetric group $\mathfrak{S}_{l}$. However, $S = S(n,p)$ is non-abelian and contains an elementary abelian $p$-subgroup of index $p$, namely $A = A(n,p)$. By \cite[Proposition 6.2.3]{Cl}, if $S$ is isomorphic to a Sylow $p$-subgroup of some symmetric group $\mathfrak{S}_{l}$, then $|S| \geq p^{p+1}$, which contradicts $|S| = p^{n+2}$ (recall that $n \leq p-3$). If $G = 2.\mathfrak{A}_{m}$, note that $S$ is a defect group of a $p$-block of $2.\mathfrak{A}_{m}$ and that $2.\mathfrak{A}_{m} \trianglelefteq 2.\mathfrak{S}_{m}$(see \cite[2.7.2]{Wil}). Hence $S$ is isomorphic to a defect group of a $p$-block of the group $2.\mathfrak{S}_{m}$ (see \cite[Chapter 5, Theorem 5.16 (ii)]{NT}). By Lemma 3.10, $S$ is isomorphic to a Sylow $p$-subgroup of some symmetric group $\mathfrak{S}_{l}$. and a completely analogous argument, a contradiction is also obtained.

Next, assume that $G/Z(G)$ is a group of  Lie type. By Proposition 3.14, we see that case in cross characteristic is not possible. Therefore, we assume that $G/Z(G)$ is defined over a field of characteristic $p$. By Lemma 3.16, $S$ is isomorphic to a Sylow $p$-subgroup of $G/Z(G)$. However, by Lemma 3.19, $S$ is not isomorphic to any Sylow $p$-subgroup of a finite simple group of Lie type defined over a field of characteristic $p$, which is a contradiction. 

Finally, assume that $G/Z(G)$ is a sporadic simple group. The order of the Schur multiplier of every sporadic simple group is coprime to $p$ \cite[Table 5.3]{GLS}. By Remark 3.7, we compare the exponent of $p$ in $|S|$ with those in the orders of the various sporadic simple groups. This comparison shows that $S$ cannot be a defect group of $G$, which yields a contradiction. 
\end{proof}

The desired results follow immediately from Theorem 1.2.

\begin{proof}[Proof of Theorem 1.3 and Theorem 1.4]
This follows directly from Lemma 2.1 and Theorem 1.2.
\end{proof}

\end{document}